\documentclass[reqno,12pt]{amsart}

\usepackage{amsmath, epsfig, cite}
\usepackage{amssymb}
\usepackage{amsfonts}
\usepackage{latexsym}
\usepackage{amsthm}

\newtheorem{theorem}{Theorem}

\newtheorem{corollary}[theorem]{Corollary}
\newtheorem{conjecture}{Conjecture}
\newtheorem{lemma}{Lemma}

\theoremstyle{remark}

\numberwithin{equation}{section}

\allowdisplaybreaks

\author{Victor J.\ W.\ Guo}
\address{School of Mathematics and Statistics, Huaiyin Normal University,
Huai'an 223300, Jiangsu, People's Republic of China}
\email{jwguo@hytc.edu.cn}
\thanks{The first author was partially supported by the National Natural
Science Foundation of China (grant 11771175).}

\author{Michael J.\ Schlosser}
\address{Fakult\"at f\"ur Mathematik, Universit\"at Wien,
Oskar-Morgenstern-Platz~1, A-1090 Vienna, Austria}
\email{michael.schlosser@univie.ac.at}
\thanks{The second author was partially supported by FWF Austrian Science
Fund grant P 32305.}

\title[Some $q$-congruences for truncated basic series: even powers]{Some
new $q$-congruences for truncated\\ basic hypergeometric series: even powers}



\begin{document}

\begin{abstract}
We provide several new $q$-congruences for truncated basic hypergeometric
series with the base being an even power of $q$.  Our results mainly concern
congruences modulo the square or the cube of a cyclotomic polynomial and
complement corresponding ones of an earlier paper containing
$q$-congruences for truncated basic hypergeometric
series with the base being an odd power of $q$.
We also give a number of related conjectures including $q$-congruences modulo
the fifth power of a cyclotomic polynomial and a
congruence for a truncated ordinary hypergeometric series
modulo the seventh power of a prime greater than $3$.
\end{abstract}

\maketitle

\section{Introduction}
In his first letter to Hardy from 1913, Ramanujan announced that
(cf.\ \cite[p.~25, Equation~(2)]{BR})
\begin{align}
\sum_{k=0}^\infty(8k+1)\frac{(\frac{1}{4})_k^4}{k!^4}
=\frac{2\sqrt{2}}{\sqrt{\pi}\,\Gamma(\frac 34)^2}, \label{eq:4ram}
\end{align}
along with similar hypergeometric identities.
Here $(a)_n=a(a+1)\cdots(a+n-1)$ denotes the Pochhammer symbol.
He did not provide proofs. This identity was eventually proved by Hardy
in \cite[p.~495]{Ha}.

In 1997, Van Hamme \cite{Hamme} proposed 13 interesting $p$-adic
analogues of\break
Ramanujan-type formulas for $1/\pi$ \cite{Ramanujan}, such as
\begin{align}
\sum_{k=0}^{(p-1)/4}(8k+1)\frac{(\frac{1}{4})_k^4}{k!^4}
\equiv p\frac{\Gamma_p(\frac 12)\Gamma_p(\frac 14)}{\Gamma_p(\frac 34)}
\pmod{p^3},\quad\text{if $p\equiv 1\pmod{4}$},
\label{eq:pram}
\end{align}
where $p$ is an odd prime and $\Gamma_p$ is the $p$-adic gamma
function~\cite{Mor}.
Van Hamme \cite{Hamme} himself proved three of them.
Nowadays all of the 13 supercongruences have been confirmed
by different techniques (see \cite{Lo,Mo,McCO,OZ,Swisher}).
For some informative background on Ramanujan-type supercongruences,
we refer the reader to Zudilin's paper \cite{Zud2009}.
During the past few years, congruences and supercongruences
have been generalized to the $q$-world by many authors (see, for example,
\cite{Gorodetsky,Guo2018,Guo2,Guo-t,Guo-gz,Guo-par,
GJZ,GL18,GPZ,GS19,GS0,GS,GuoZu,NP,Straub,Tauraso2}).
As explained in \cite{GuoZu}, $q$-supercongruences are closely related to
studying the asymptotic behaviour of $q$-series at roots of unity.

Recently, the authors \cite[Theorems 1 and 2]{GS19} proved that
for odd $d\geqslant 5$,
\begin{equation}
\sum_{k=0}^{n-1}[2dk+1]\frac{(q;q^d)_k^d}{(q^d;q^d)_k^d}q^{\frac{d(d-3)k}{2}}
\equiv
\begin{cases} 0\pmod{\Phi_n(q)^2}, &\text{if $n\equiv -1\pmod{d}$,}\\[5pt]
0\pmod{\Phi_n(q)^3}, &\text{if $n\equiv -\frac{1}{2}\pmod{d}$,}
\end{cases} \label{odd-1}
\end{equation}
and for odd $d\geqslant 3$ and $n>1$,
\begin{equation}
\sum_{k=0}^{n-1}[2dk-1]\frac{(q^{-1};q^d)_k^d}{(q^d;q^d)_k^d}
q^{\frac{d(d-1)k}{2}}
\equiv
\begin{cases} 0\pmod{\Phi_n(q)^2}, &\text{if $n\equiv 1\pmod{d}$,}\\[5pt]
0\pmod{\Phi_n(q)^3}, &\text{if $n\equiv \frac{1}{2}\pmod{d}$.}
\end{cases}  \label{odd-2}
\end{equation}
Here and throughout the paper, we adopt the standard $q$-notation:
For an indeterminate $q$, let
\begin{equation*}
(a;q)_n=(1-a)(1-aq)\cdots (1-aq^{n-1})
\end{equation*}
be the {\em $q$-shifted factorial}.
For convenience, we compactly write
\begin{equation*}
(a_1,a_2,\ldots,a_m;q)_n=(a_1;q)_n (a_2;q)_n\cdots (a_m;q)_n
\end{equation*}
for a product of $q$-shifted factorials.
Moreover,
\begin{equation*}
[n]=[n]_q=1+q+\cdots+q^{n-1}
\end{equation*}
denotes the {\em $q$-integer}, which can be defined by $[n]=(q^n-1)/(q-1)$
to hold for any integer $n$, including negative $n$, which in particular
gives $[-1]=-1/q$ (which is needed in the $k=0$ terms of \eqref{eq:t1-2}
and \eqref{eq:new-2} and at other places in this paper).
Furthermore, $\Phi_n(q)$ denotes the $n$-th {\em cyclotomic polynomial}
in $q$, which may be defined as
\begin{align*}
\Phi_n(q)=\prod_{\substack{1\leqslant k\leqslant n\\ \gcd(n,k)=1}}(q-\zeta^k),
\end{align*}
where $\zeta$ is an $n$-th primitive root of unity.

In this paper, we shall prove results similar to \eqref{odd-1} and
\eqref{odd-2} for \textit{even} $d$. The first result concerns the case $d=2$.

\begin{theorem}\label{thm:1}
Let $n$ be an odd integer greater than $1$. Then
\begin{align}
\sum_{k=0}^{n-1}[4k+1]\frac{(q;q^2)_k^2}{(q^2;q^2)_k^2} q^{-k}
\equiv q[n]^2\pmod{[n]^2\Phi_n(q)}, \label{eq:t1-1}\\[5pt]
\sum_{k=0}^{n-1}[4k-1]\frac{(q^{-1};q^2)_k^2}{(q^2;q^2)_k^2} q^{k}
\equiv -[n]^2\pmod{[n]^2\Phi_n(q)}.  \label{eq:t1-2}
\end{align}
\end{theorem}

\begin{theorem}\label{thm:2}
Let $d\geqslant 4$ be an even integer and let $n$ be a positive integer
with $n\equiv -1\pmod{d}$. Then
\begin{equation}
\sum_{k=0}^{n-1}[2dk+1]\frac{(q;q^d)_k^d}{(q^d;q^d)_k^d}q^{\frac{d(d-3)k}{2}}
\equiv 0\pmod{\Phi_n(q)^2}.  \label{eq:new-1}
\end{equation}
\end{theorem}

\begin{theorem}\label{thm:3}
Let $d\geqslant 4$ be an even integer and let $n>1$ be an integer
with $n\equiv 1\pmod{d}$. Then
\begin{equation}
\sum_{k=0}^{n-1}[2dk-1]\frac{(q^{-1};q^d)_k^d}{(q^d;q^d)_k^d}q^{\frac{d(d-1)k}{2}}
\equiv 0\pmod{\Phi_n(q)^2}.  \label{eq:new-2}
\end{equation}
\end{theorem}
Although neither \eqref{eq:new-1} nor \eqref{eq:new-2} holds modulo
$\Phi_n(q)^3$ in general, we have the following common refinement
of \eqref{odd-1} and \eqref{eq:new-1}.
\begin{theorem}\label{thm:d=n+1}
Let $d\geqslant 4$ be an integer and let $n$ be a positive integer
with $n\equiv -1\pmod{d}$. Then
\begin{equation}
\sum_{k=0}^{n-1}[2dk+1]\frac{(q;q^d)_k^d}{(q^d;q^d)_k^d}q^{\frac{d(d-3)k}{2}}
\equiv 0\pmod{\Phi_n(q)^2\Phi_{dn-n}(q)}.  \label{eq:2nk}
\end{equation}
\end{theorem}

Let $n=p$ be an odd prime and $d=p+1$ in \eqref{eq:2nk}.
Then letting $q\to 1$, we are led to
\begin{align}
\sum_{k=0}^{p-1}(2p+2k+1)\frac{(\frac{1}{p+1})_k^{p+1}}{k!^{p+1}}
\equiv 0\pmod{p^3}.  \label{eq:2p2k}
\end{align}
Note that Sun \cite[Theorem~1.2]{Sun} proved that for any prime $p>3$
\begin{align}
\sum_{k=0}^{p-1}\frac{(\frac{1}{p+1})_k^{p+1}}{k!^{p+1}}
\equiv 0\pmod{p^5}, \label{eq:sun}
\end{align}
which also holds modulo $p^3$ for $p=3$. Substituting \eqref{eq:sun}
into \eqref{eq:2p2k}, we arrive at the following conclusion.
\begin{corollary}\label{cor:five}
Let $p$ be an odd prime. Then
\begin{align*}
\sum_{k=0}^{p-1}k\frac{(\frac{1}{p+1})_k^{p+1}}{k!^{p+1}}\equiv 0\pmod{p^3}.
\end{align*}
\end{corollary}
This result is actually a special case of
\begin{align}\label{gao}
\sum_{k=0}^{p-1}k\frac{(\frac{1}{p+1})_k^{p+1}}{k!^{p+1}}\equiv
\frac{p^3}4-\frac{p^4}8\pmod{p^5},
\end{align}
that was conjectured by Sun and was subsequently proved by
Gao~\cite{Gao} in her master thesis. See the discussion around
Equation (1.3) in  Wang's paper \cite{Wang} where \eqref{gao} is
further generalized to a congruence modolo $p^6$ for $p>5$ that involves
Bernoulli numbers.
In Section~\ref{sec:final} we propose an extension of
Corollary~\ref{cor:five} which contains additional factors in the summand
(see Conjecture~\ref{c1}).


The paper is organized as follows. We shall prove Theorem~\ref{thm:1}
in Section~\ref{sec:thm1} based on two $q$-series identities.
Theorems~\ref{thm:2} and \ref{thm:3} will be proved by giving a
common generalization of them in Section~\ref{sec:thms23}.
To accomplish this we shall make a careful use of Andrews' multiseries
generalization of the Watson transformation~\cite[Theorem~4]{Andrews75}
(which was already used in \cite{GJZ} to prove some $q$-analogues of
Calkin's congruence~\cite{Calkin}, and which was also applied in \cite{GS19}
for proving some analogous results involving the base being odd powers of $q$).
We shall prove Theorem~\ref{thm:d=n+1} by using a certain anti-symmetry
of the $k$-th summand on the left-hand side of \eqref{eq:2nk} in
Section~\ref{sec:thm4}.
Finally, in Section~\ref{sec:final} we give some concluding remarks and state
some open problems. These include some conjectural $q$-congruences modulo
the fifth power of a cyclotomic polynomial and congruences for truncated
ordinary hypergeometric series, one of them, see \eqref{eq:p7},
modulo the seventh power of a prime greater than $3$.

We would like to thank the two anonymous referees for their comments.
We especially thank the second referee for her or his detailed list of constructive
suggestions for improvement of the paper.

\section{Proof of Theorem \ref{thm:1}}\label{sec:thm1}
It is easy to prove by induction on $n$ that
\begin{align*}
\sum_{k=0}^{n-1}[4k+1]\frac{(q;q^2)_k^2}{(q^2;q^2)_k^2} q^{-k}
=[n]^2(1+q^n)^2\frac{(q;q^2)_n^2}{(q^2;q^2)_n^2}q^{1-n}
=[n]^2\begin{bmatrix}2n-1\\n-1\end{bmatrix}^2\frac{q^{1-n}}{(-q;q)_{n-1}^4}.
\end{align*}
Since $q^n\equiv 1\pmod{\Phi_n(q)}$, the proof of \eqref{eq:t1-1}
then follows from the fact
$$
\begin{bmatrix}2n-1\\n-1\end{bmatrix}=
\prod_{k=1}^{n-1}\frac{1-q^{2n-k}}{1-q^k}
\equiv q^{-\binom n2}(-1)^{n-1}\equiv 1\pmod{\Phi_n(q)}
$$
for odd $n$ and $(-q;q)_{n-1}\equiv 1\pmod{\Phi_n(q)}$
(see, for example, \cite[Equation~(2.3)]{Guo2}).

Similarly, we can prove by induction that
\begin{align*}
\sum_{k=0}^{n-1}[4k-1]\frac{(q^{-1};q^2)_k^2}{(q^2;q^2)_k^2} q^{k}
=-[n]^2(1+q^n)^2\frac{(q^{-1};q^2)_n^2}{(q^2;q^2)_n^2}q^{n}.
\end{align*}
The proof of \eqref{eq:t1-2} then follows from that of \eqref{eq:t1-1}
and the following relation
$$
(q^{-1};q^2)_n=(q;q^2)_{n}\frac{1-q^{-1}}{1-q^{2n-1}}.
$$

We point out that, using the congruence
$(-q;q)_{(n-1)/2}^2\equiv q^{(n^2-1)/8}\pmod{\Phi_n(q)}$ for odd $n$
(see, for example, \cite[Lemma~2.1]{Guo2}), we can prove the following
similar congruences: for any odd positive integer $n>1$,
\begin{align}
\sum_{k=0}^{(n-1)/2}[4k+1]\frac{(q;q^2)_k^2}{(q^2;q^2)_k^2} q^{-k}
\equiv q[n]^2\pmod{[n]^2\Phi_n(q)}, \\[5pt]
\sum_{k=0}^{(n+1)/2}[4k-1]\frac{(q^{-1};q^2)_k^2}{(q^2;q^2)_k^2} q^{k}
\equiv -[n]^2\pmod{[n]^2\Phi_n(q)}.
\end{align}
The details of the proof are left to the interested reader.

\section{Proof of Theorems \ref{thm:2} and \ref{thm:3}}\label{sec:thms23}
We shall first prove the following unified generalization of
Theorems \ref{thm:2} and \ref{thm:3} for $d=4$.
\begin{theorem}\label{thm:d=4}
Let $r$ be an odd integer. Let $n>1$ be an odd integer with
$n\equiv -r\pmod{4}$ and $n\geqslant \max\{r,4-r\}$. Then
\begin{equation}
\sum_{k=0}^{n-1}[8k+r]\frac{(q^r;q^4)_k^4}{(q^4;q^4)_k^4}q^{(4-2r)k}
\equiv 0\pmod{\Phi_n(q)^2}.  \label{eq:new-3}
\end{equation}
\end{theorem}
\begin{proof}
Let $\alpha$, $j$ and $r$ be integers. It is easy to see that
\begin{equation*}
(1-q^{\alpha n-dj+d-r})(1-q^{\alpha  n+dj-d+r})+(1-q^{dj-d+r})^2
q^{\alpha  n-dj+d-r}=(1-q^{\alpha  n})^2
\end{equation*}
and $1-q^{\alpha n}\equiv 0\pmod{\Phi_n(q)}$, and so
\begin{equation*}
(1-q^{\alpha  n-dj+d-r})(1-q^{\alpha  n+dj-d+r})\equiv -(1-q^{dj-d+r})^2
q^{\alpha  n-dj+d-r}\pmod{\Phi_n(q)^2}.
\end{equation*}
It follows that
\begin{equation}
(q^{r-\alpha  n},q^{r+\alpha  n};q^d)_k \equiv (q^r;q^d)_k^2 \pmod{\Phi_n(q)^2}.
\label{eq:mod-square}
\end{equation}
It is clear that $3n\equiv r\pmod{4}$. Therefore,
by \eqref{eq:mod-square} and the $q\mapsto q^4$, $a\mapsto q^r$,
$b\mapsto q^r$, $c\mapsto q^{r+3n}$, $n\mapsto (3n-r)/4$
instance of the terminating ${}_6\phi_5$ summation
(see \cite[Appendix (II.21)]{GR}):
\begin{align*}
_6\phi_5
\left[\begin{array}{c}
a,\, qa^{\frac12},\, -qa^{\frac12},\, b,\, c,\, q^{-n} \\
a^{\frac12},\, -a^{\frac12},\, aq/b,\, aq/c,\, aq^{n+1}\end{array};
q,\,\frac{aq^{n+1}}{bc}
\right]
=\frac{(aq, aq/bc;q)_n}{(aq/b, aq/c;q)_n},
\end{align*}
where the {\it basic hypergeometric series $_{r+1}\phi_r$}
(see \cite{GR}) is defined as
$$
_{r+1}\phi_{r}\left[\begin{array}{c}
a_1,a_2,\ldots,a_{r+1}\\
b_1,b_2,\ldots,b_{r}
\end{array};q,\, z
\right]
=\sum_{k=0}^{\infty}\frac{(a_1,a_2,\ldots, a_{r+1};q)_k z^k}
{(q,b_1,\ldots,b_{r};q)_k},
$$
modulo $\Phi_n(q)^2$, the left-hand side of \eqref{eq:new-3} is congruent to
\begin{align}
\sum_{k=0}^{(3n-r)/4}[8k+r]\frac{(q^r,q^r,q^{r+3n}, q^{r-3n};q^4)_k }
{(q^4, q^4, q^{4-3n}, q^{4+3n};q^4)_k} q^{(4-2r)k}
=[r]\frac{(q^{r+4}, q^{4-3n-r};q^4)_{(3n-r)/4}}{(q^4, q^{4-3n};q^4)_{(3n-r)/4}}.
\label{eq:6phi5-2}
\end{align}
Note that $(3n-r)/4\leqslant n-1$ by the condition $n\geqslant 4-r$.
It is clear that $(q^{r+4};q^4)_{(3n-r)/4}$ has the factor $1-q^{3n}$,
and $(q^{4-r-3n};q^4)_{(3n-r)/4}$ has the factor $(1-q^{-2n})$
since $(3n-r)/4\geqslant (n+r)/4$ by the condition $n\geqslant r$.
Therefore the numerator on the right-hand side of \eqref{eq:6phi5-2}
is divisible by $\Phi_n(q)^2$,
while the denominator is relatively prime to $\Phi_n(q)$.
This completes the proof.
\end{proof}

We need the following lemma in our proof of
Theorems \ref{thm:2} and \ref{thm:3} for $d\geqslant 6$.
\begin{lemma}\label{lem:one}
Let $d\geqslant 5$ be an integer and let $r$ be an integer with $\gcd(d,r)=1$.
Let $n=ad-r\geqslant r$ with $a\geqslant 1$.
Suppose that $2r+kd\equiv 0\pmod{n}$ for some $k>0$. Then $k\geqslant a(d-4)/2$.
\end{lemma}
\begin{proof}Since $\gcd(d,r)=1$, we have $\gcd(n,d)=1$ for $n=ad-r$.
Noticing that
$$
2r+kd=(k+2a)d-2(ad-r)=(k+2a)d-2n,
$$
we conclude that $(k+2a)d\equiv 0\pmod{n}$.
It follows that $k+2a$ is a multiple of $n$ and so $k+2a\geqslant n$, i.e.,
$$
k+2a\geqslant ad-r.
$$
By the condition $ad-r\geqslant r$ in the lemma, we get $r\leqslant ad/2$.
Substituting this into the above inequality, we obtain the desired result.
\end{proof}

We now give a common generalization of Theorems \ref{thm:2} and \ref{thm:3}.
\begin{theorem}\label{thm:5}
Let $d\geqslant 4$ be an even integer and let $r$ be an integer with
$\gcd(d,r)=1$. Let $n>1$ be an integer with $n\equiv -r\pmod{d}$
and $n\geqslant \max\{r,d-r\}$. Then
\begin{equation}
\sum_{k=0}^{n-1}[2dk+r]\frac{(q^r;q^d)_k^d}{(q^d;q^d)_k^d}q^{\frac{d(d-r-2)k}{2}}
\equiv 0\pmod{\Phi_n(q)^2}.  \label{eq:new-d}
\end{equation}
\end{theorem}
\begin{proof}The $d=4$ case is just Theorem \ref{thm:d=4}.
We now suppose that $d\geqslant 6$. The proof of this case is
intrinsically the same as that of Theorem \ref{thm:d=4}.
Here we need to use a complicated transformation formula due to
Andrews \cite[Theorem~4]{Andrews75}:
\begin{align}
&\sum_{k\geqslant 0}\frac{(a,q\sqrt{a},-q\sqrt{a},b_1,c_1,\dots,b_m,c_m,q^{-N};q)_k}
{(q,\sqrt{a},-\sqrt{a},aq/b_1,aq/c_1,\dots,aq/b_m,aq/c_m,aq^{N+1};q)_k}
\left(\frac{a^mq^{m+N}}{b_1c_1\cdots b_mc_m}\right)^k \notag\\[5pt]
&\quad=\frac{(aq,aq/b_mc_m;q)_N}{(aq/b_m,aq/c_m;q)_N}
\sum_{l_1,\dots,l_{m-1}\geqslant 0}
\frac{(aq/b_1c_1;q)_{l_1}\cdots(aq/b_{m-1}c_{m-1};q)_{l_{m-1}}}
{(q;q)_{l_1}\cdots(q;q)_{l_{m-1}}} \notag\\[5pt]
&\quad\quad\times\frac{(b_2,c_2;q)_{l_1}\dots(b_m,c_m;q)_{l_1+\dots+l_{m-1}}}
{(aq/b_1,aq/c_1;q)_{l_1}
\dots(aq/b_{m-1},aq/c_{m-1};q)_{l_1+\dots+l_{m-1}}} \notag\\[5pt]
&\quad\quad\times\frac{(q^{-N};q)_{l_1+\dots+l_{m-1}}}
{(b_mc_mq^{-N}/a;q)_{l_1+\dots+l_{m-1}}}
\frac{(aq)^{l_{m-2}+\dots+(m-2)l_1} q^{l_1+\dots+l_{m-1}}}
{(b_2c_2)^{l_1}\cdots(b_{m-1}c_{m-1})^{l_1+\dots+l_{m-2}}},  \label{andrews}
\end{align}
which is a multiseries generalization of Watson's $_8\phi_7$ transformation
formula (see \cite[Appendix (III.18)]{GR}):
\begin{align}
& _{8}\phi_{7}\!\left[\begin{array}{cccccccc}
a,& qa^{\frac{1}{2}},& -qa^{\frac{1}{2}}, & b,    & c,    & d,    & e,    & q^{-n} \\
  & a^{\frac{1}{2}}, & -a^{\frac{1}{2}},  & aq/b, & aq/c, & aq/d, & aq/e, & aq^{n+1}
\end{array};q,\, \frac{a^2q^{n+2}}{bcde}
\right] \notag\\[5pt]
&\quad =\frac{(aq, aq/de;q)_n}
{(aq/d, aq/e;q)_n}
\,{}_{4}\phi_{3}\!\left[\begin{array}{c}
aq/bc,\ d,\ e,\ q^{-n} \\
aq/b,\, aq/c,\, deq^{-n}/a
\end{array};q,\, q
\right].  \label{eq:8phi7}
\end{align}

It is easy to see that $(d-1)n\equiv r\pmod{d}$. Hence,
by \eqref{eq:mod-square}, modulo $\Phi_n(q)^2$, the left-hand side of
\eqref{eq:new-d} is congruent to
\begin{align*}
\sum_{k=0}^{(dn-n-r)/d}[r]\frac{(q^r,q^d\sqrt{q^r},-q^d\sqrt{q^r},
\overbrace{q^r,\ldots,q^r}^{\text{$(d-3)$'s $q^r$}},q^{r+(d-1)n},q^{r-(d-1)n};q^d)_k}
{(q^d,\sqrt{q^r},-\sqrt{q^r},q^d,\ldots,q^d,q^{d-(d-1)n},q^{d+(d-1)n};q^d)_k}
q^{\frac{d(d-r-2)k}{2}},
\end{align*}
where we have used the fact $(dn-n-r)/d\leqslant n-1$ by the condition
$n\geqslant d-r$. Furthermore, by the $q\mapsto q^d$, $a\mapsto q^r$,
$b_i\mapsto q^r$, $c_i\mapsto q^r$, for $1\leqslant i\leqslant m-1$,
$b_m\mapsto q^r$, $c_m\mapsto q^{r+(d-1)n}$, $N\mapsto ((d-1)n-r)/d$
case of Andrews' transformation \eqref{andrews},
the above summation can be written as
\begin{align}
&[r]\frac{(q^{d+r},q^{d+n-dn-r};q^d)_{(dn-n-r)/d}}{(q^d,q^{d+n-dn};q^d)_{(dn-n-r)/d}}
\sum_{l_1,\dots,l_{m-1}\geqslant 0}
\frac{(q^{d-r};q^d)_{l_1}\cdots(q^{d-r};q^d)_{l_{m-1}}}
{(q^d;q^d)_{l_1}\cdots(q^d;q^d)_{l_{m-1}}} \notag\\[5pt]
&\quad\quad\times\frac{(q^r,q^r;q^d)_{l_1}\dots
(q^r,q^{r+(d-1)n};q^d)_{l_1+\dots+l_{m-1}}}
{(q^d,q^d;q^d)_{l_1}
\dots(q^d,q^d;q^d)_{l_1+\dots+l_{m-1}}} \notag\\[5pt]
&\quad\quad\times\frac{(q^{r-(d-1)n};q^d)_{l_1+\dots+l_{m-1}}}
{(q^{2r};q^d)_{l_1+\dots+l_{m-1}}}
\frac{q^{(d+r)(l_{m-2}+\dots+(m-2)l_1)} q^{d(l_1+\dots+l_{m-1})}}
{q^{2rl_1}\cdots q^{2r(l_1+\dots+l_{m-2})}},  \label{eq:multi}
\end{align}
where $m=(d-2)/2$.

It is easy to see that $(q^{d+r};q^d)_{(dn-n-r)/d}$ contains the
factor $1-q^{(d-1)n}$. Similarly, $(q^{d+n-dn-r};q^d)_{(dn-n-r)/d}$
contains the factor $1-q^{(2-d)n}$ since $(dn-n-r)/d\geqslant (n+r)/d$
by the conditions $d\geqslant 6$ and $n\geqslant r$.
Thus, the expression $(q^{d+r},q^{d+n-dn-r};q^d)_{(dn-n-r)/d}$
in the fraction before the multiple summation  is divisible by $\Phi_n(q)^2$.

Note that the non-zero terms in the multiple summation of
\eqref{eq:multi} are just those indexed by $(l_1,\ldots,l_{m-1})$
with $l_1+\dots+l_{m-1}\leqslant (dn-n-r)/d\leqslant n-1$ because of
the factor $(q^{r-(d-1)n};q^d)_{l_1+\dots+l_{m-1}}$ in the
numerator. This immediately implies that all the other
$q$-factorials in the denominator of the multiple summation of
\eqref{eq:multi} do not contain factors of the form $1-q^{\alpha n}$
(and are therefore relatively prime to $\Phi_n(q)$), except for
$(q^{2r};q^d)_{l_1+\dots+l_{m-1}}$. If $n=d-r$, then it is clear
that at least one $(q^{d-r};q^d)_{l_i}$ contains the factor $1-q^n$ for $l_1+\dots+l_{m-1}>0$.
We now assume that $n\geqslant 2d-r$ and so $n>\max\{d,r\}$ in this case.
Thus, if $(q^{2r};q^d)_{l_1+\dots+l_{m-1}}$ has a factor $1-q^{kn}$,
then the number $k$ is unique since $l_1+\dots+l_{m-1}\leqslant n-1$
and $\gcd(n,d)=1$. Moreover, if such a $k$ exists, then we must have
$k\geqslant a(d-4)/2$ by Lemma \ref{lem:one}, where $n=ad-r$. It
follows that $l_1+\dots+l_{m-1}\geqslant k$ and at least one $l_i$ is greater
than or equal to $k/(m-1)=2k/(d-4)\geqslant a$ and so $(q^{d-r};q^d)_{l_i}$ contains the
factor $1-q^n$ in this case. This proves that the denominator of the
reduced form of the multiple summation of \eqref{eq:multi} is always
relatively prime to $\Phi_n(q)$, which completes the proof of
\eqref{eq:new-d}.
\end{proof}

\section{Proof of Theorem \ref{thm:d=n+1}}\label{sec:thm4}
We shall prove
\begin{equation*}
\sum_{k=0}^{n-1}[2dk+1]\frac{(q;q^d)_k^d}{(q^d;q^d)_k^d}q^{\frac{d(d-3)k}{2}}
\equiv 0\pmod{\Phi_{dn-n}(q)},  
\end{equation*}
which is equivalent to
\begin{equation}
\sum_{k=0}^{(dn-n-1)/d}[2dk+1]\frac{(q;q^d)_k^d}{(q^d;q^d)_k^d}q^{\frac{d(d-3)k}{2}}
\equiv 0\pmod{\Phi_{dn-n}(q)},  \label{eq:2nk-3}
\end{equation}
because $(q;q^d)_k$ has the factor $1-q^{dn-n}$ and is therefore divisible
by $\Phi_{dn-n}(q)$
for $(dn-n-1)/d<k\leqslant n-1$, while $(q^d;q^d)_k$ is coprime with
$\Phi_{dn-n}(q)$ for these $k$.

Since $q^{dn-n}\equiv 1\pmod{\Phi_{dn-n}(q)}$, we have
\begin{align}\label{aqcong}
\frac{(q;q^d)_{(dn-n-1)/d} }{(q^d;q^d)_{(dn-n-1)/d}}
&=\frac{(1-q)(1-q^{d+1})\cdots (1-q^{dn-n-d})}
{(1-q^{d})(1-q^{2d})\cdots (1-q^{dn-n-1})} \notag\\[5pt]
&\equiv \frac{(1-q)(1-q^{d+1})\cdots (1-q^{dn-n-d})}
{(1-q^{-(dn-n-d)})(1-q^{-(dn-n-2d)})\cdots (1-q^{-1})}\notag\\[5pt]
&=(-1)^{\frac{dn-n-1}{d}}q^{\frac{(d-1)(n-1)(dn-n-1)}{2d}} \pmod{\Phi_{dn-n}(q)}.
\end{align}
Furthermore, for $0\leqslant k\leqslant (dn-n-1)/d$, we have
\begin{align*}
&\frac{(q;q^d)_{(dn-n-1)/d-k} }{(q^d;q^d)_{(dn-n-1)/d-k}}\\[5pt]
&\quad=\frac{(q;q^d)_{(dn-n-1)/d} }{(q^d;q^d)_{(dn-n-1)/d}}
\frac{(1-q^{dn-n-1-(k-1)d})(1-q^{dn-n-1-(k-2)d})\cdots (1-q^{dn-n-1})}
{(1-q^{dn-n-kd})(1-q^{dn-n-(k-1)d})\cdots (1-q^{dn-n-d})}
\\[5pt]
&\quad\equiv (-1)^{\frac{dn-n-1}{d}}q^{\frac{(d-1)(n-1)(dn-n-1)}{2d}}
\frac{(1-q^{-1-(k-1)d})(1-q^{-1-(k-2)d})\cdots (1-q^{-1})}
{(1-q^{-kd})(1-q^{-(k-1)d})\cdots (1-q^{-d})} \\[5pt]
&\quad=(-1)^{\frac{dn-n-1}{d}}q^{\frac{(d-1)(n-1)(dn-n-1)}{2d}+(d-1)k}
\frac{(q;q^d)_{k} }{(q^d;q^d)_{k}} \pmod{\Phi_{dn-n}(q)}.
\end{align*}
Taking  the most left- and right-hand sides of this congruence
to the power $d$, it follows, using
$q^{dn-n}\equiv 1\pmod{\Phi_{dn-n}(q)}$,
that for $0\leqslant k\leqslant (dn-n-1)/d$ there holds
\begin{align*}
&[2d((dn-n-1)/d-k)+1]\frac{(q;q^d)_{(dn-n-1)/d-k}^d}
{(q^d;q^d)_{(dn-n-1)/d-k}^d}q^{\frac{d(d-3)((dn-n-1)/d-k)}{2}}\\[5pt]
&\quad \equiv (-1)^{dn-n}q^{\frac{(dn-n)(dn-n-3)}{2}} [2dk+1]
\frac{(q;q^d)_k^d}{(q^d;q^d)_k^d}q^{\frac{d(d-3)k}{2}} \pmod{\Phi_{dn-n}(q)}.
\end{align*}
It is easy to check that
$(-1)^{dn-n}q^{\frac{(dn-n)(dn-n-3)}{2}}\equiv -1\pmod{\Phi_{dn-n}(q)}$
whenever $dn-n$ is odd or even.
This proves that the $k$-th and $((dn-n-1)/d-k)$-th terms of the
left-hand side of \eqref{eq:2nk-3} cancel each other modulo $\Phi_{dn-n}(q)$
and therefore \eqref{eq:2nk-3} holds. Equivalently, \eqref{eq:2nk} holds
modulo $\Phi_{dn-n}(q)$. Moreover, by \eqref{odd-1} and \eqref{eq:new-1},
one sees that \eqref{eq:2nk} also holds modulo $\Phi_{n}(q)^2$ for
$d\geqslant 4$. The proof then follows from the fact $\Phi_{n}(q)^2$ and
$\Phi_{dn-n}(q)$ are relatively prime polynomials.

\section{Concluding remarks and open problems}\label{sec:final}

Having establishing ($q$-)congruences for truncated (basic) hypergeometric
series, one can wonder what their `archimedian' analogues are, i.e.\ whether
the infinite sums from $k=0$ to $\infty$ have known evaluations,
just as \eqref{eq:4ram} is such an archimedian analogue for \eqref{eq:pram}.

In many cases of our results,
especially when dealing with arbitrary exponents $d$,
we are not aware of explicit evaluations in the archimedian case.
However for small $d$ we can easily find corresponding evaluations
by suitably specializing known summations for (basic) hypergeometric series,
such as Rogers' nonterminating $_6\phi_5$ summation (cf.\
\cite[Appendix~(II.20)]{GR}),
\begin{equation}\label{eq:6phi5}
{}_{6}\phi_5\!\left[\begin{array}{cccccc}
a,& qa^{\frac{1}{2}},& -qa^{\frac{1}{2}}, & b,    & c,    & d \\
  & a^{\frac{1}{2}}, & -a^{\frac{1}{2}},  & aq/b, & aq/c, & aq/d
\end{array};q,\, \frac{aq}{bcd}
\right]
 =\frac{(aq, aq/bc,aq/bd,aq/cd;q)_\infty}
{(aq/b,aq/c,aq/d,aq/bcd;q)_\infty},
\end{equation}
where $|aq/bcd|<1$ for convergence.

Indeed, by replacing $q$ by $q^4$, and letting $a=b=c=d=q^r$,
we obtain from \eqref{eq:6phi5}, after multiplying both sides
by $[r]$, the following identity:
\begin{align}\label{eq:rdid}
\sum_{k\ge 0}[8k+r]\frac{(q^r;q^4)_k^4}{(q^4;q^4)_k^4}q^{(4-2r)k}
&=[r]\frac{(q^{4+r},q^{4-r},q^{4-r},q^{4-r};q^4)_\infty}
{(q^4,q^4,q^4,q^{4-2r};q^4)_\infty}\notag\\
&=\frac{[r]\,\Gamma_{q^4}(1-\frac r2)}
{\Gamma_{q^4}(1+\frac r4)\Gamma_{q^4}(1-\frac r4)^3},
\end{align}
valid for $r<2$.
In the last equation we have rewritten the product using
the $q$-Gamma function
\begin{equation*}
\Gamma_q(x)=\frac{(q;q)_\infty}{(q^x;q)_\infty}(1-q)^{1-x},
\end{equation*}
defined for $0<q<1$ (cf.\ \cite[Section~1.10]{GR}).
For $r\leqslant1$ being an odd integer, we have thus just
established an archimedian analogue of Theorem~\ref{thm:d=4}.

Now, since $\lim_{q\to 1^-}\Gamma_q(x)=\Gamma(x)$, we obtain that in the
$q\to 1^-$ limit \eqref{eq:rdid} becomes
\begin{align}\label{eq:rdidqto1}
\sum_{k\ge 0}(8k+r)\frac{(\frac r4)_k^4}{k!^4}
=\frac{r\,\Gamma(1-\frac r2)}
{\Gamma(1+\frac r4)\Gamma(1-\frac r4)^3}
=\frac{4\,\sin(\frac{r\pi}4)\,\Gamma(1-\frac r2)}
{\pi\Gamma(1-\frac r4)^2},
\end{align}
where we have used the well-known reflection formula for the Gamma function.
It is now immediate that for $r=1$ we get \eqref{eq:4ram}
while for $r=-1$ we get the similarly attractive evaluation
\begin{align}
\sum_{k=0}^\infty(8k-1)\frac{(\frac{-1}{4})_k^4}{k!^4}
=\frac{-16\sqrt{2}}{\sqrt{\pi}\,\Gamma(\frac 14)^2}. \label{eq:4ram-1}
\end{align}

Many other identities involving $\pi$ can similarly be obtained.
At this place, in passing, we would like to point out that
by replacing $q$ by $q^2$ in \eqref{eq:6phi5} and putting $a=q^2$,
$b=c=d=q$ one readily obtains
\begin{align}
\sum_{k\ge 0}\frac{1+q^{2k+1}}{1+q}\frac{(1-q)^2}{(1-q^{2k+1})^2}q^k
=\frac{(q^4,q^2,q^2,q^2;q^2)_\infty}{(q^3,q^3,q^3,q;q^2)_\infty},
\end{align}
which, as recently noted by Sun~\cite[Equation~(1.3)]{Sun2}
(who derived this identity by completely different means)
is easily seen to be a $q$-analogue of Euler's identity
$$
\sum_{k\ge 0}\frac 1{(2k+1)^2}=\frac{\pi^2}8,
$$
used to prove his famous evaluation $\zeta(2)=\pi^2/6$.
See \cite{Wei} for recent new samples of expansions
involving $\pi$, obtained by suitably specializing $q$-series
identities.

\smallskip
We turn to discussing whether some of the results obtained in the paper
can be further strengthened.
We have proved Theorems \ref{thm:2} and \ref{thm:3} by establishing a
common generalization of them, namely Theorem~\ref{thm:5}.
However, we are unable to
prove a similar common generalization of \eqref{odd-1} and \eqref{odd-2}.
Numerical calculation for $q=1$ suggests that
there are no congruences for the left-hand side of \eqref{eq:new-d}
with \textit{odd} $d\geqslant 5$ that would hold in general (in particular,
the case $d=5$ and $r=3$ appears to be such a counterexample).

Nevertheless, we would like to give the following result being similar
to Theorem~\ref{thm:5}.
\begin{theorem}\label{thm:6}
Let $d\geqslant 5$ be an odd integer and let $r$ be an even integer with
$\gcd(d,r)=1$. Let $n>1$ be an odd integer with $n\equiv -r\pmod{d}$
and $n\geqslant \max\{r,d-r\}$. Then
\begin{equation}
\sum_{k=0}^{n-1}[2dk+r]_{q^2}\frac{(q^{2r};q^{2d})_k^d}{(q^{2d};q^{2d})_k^d}q^{d(d-r-2)k}
\equiv 0\pmod{\Phi_n(q)^2}.  \label{eq:new-odd}
\end{equation}
\end{theorem}
The proof of Theorem \ref{thm:6} is similar to that of Theorem \ref{thm:5}.
In this case we need to apply Andrews' transformation \eqref{andrews} with
$m=(d-1)/2$,   $q\to q^{2d}$, $a=q^{2r}$, $b_1=q^{d+r}$, $b_2=\ldots=b_{m}=q^{2r}$,
$c_1=\cdots=c_{m-1}=q^{2r}$, $c_m={q^{2r+2(d-1)n}}$ and $N=((d-1)n-r)/d$.
The details of the proof are omitted here.

We can also prove the following refinement of \eqref{odd-2} and
\eqref{eq:new-2}. However, we are unable to deduce any interesting
conclusion similar to \eqref{eq:2p2k} from this result by letting $q\to 1$.
\begin{theorem}\label{thm:7}
Let $d\geqslant 3$ be an integer and let $n>1$ be an  integer
with $n\equiv 1\pmod{d}$. Then
\begin{equation*}
\sum_{k=0}^{n-1}[2dk-1]\frac{(q^{-1};q^d)_k^d}{(q^d;q^d)_k^d}q^{\frac{d(d-1)k}{2}}
\equiv 0\pmod{\Phi_n(q)^2\Phi_{dn-n}(q)}.
\end{equation*}
\end{theorem}

We would like to propose the following three conjectures
which are similar to Corollary~\ref{cor:five}.
\begin{conjecture}\label{c1}
Let $r$ be a positive integer and let $p$ be a
prime with $p>2r+1$. Then
\begin{align*}
\sum_{k=0}^{p-1}k^r \left(k+\frac{1}{p+1}\right)^r
\frac{(\frac{1}{p+1})_k^{p+1}}{k!^{p+1}}\equiv 0\pmod{p^4}.
\end{align*}
\end{conjecture}

\begin{conjecture}\label{c7}
Let $p>3$ be a prime. Then
\begin{equation}
\sum_{k=0}^{p-1}(2pk+2k+1)\frac{(\frac{1}{p+1})_k^{2p+2}}{k!^{2p+2}}
\equiv 0\pmod{p^7}. \label{eq:p7}
\end{equation}
More generally, if $p>3$ is a prime and $r$ a positive integer, then
\begin{equation}
\sum_{k=0}^{p^r-1}\Big(2k\frac{p^{r+1}-1}{p^r-1}+1\Big)
\frac{\big(\frac{p^r-1}{p^{r+1}-1}\big)_k^{2\frac{p^{r+1}-1}{p-1}}}
{k!^{2\frac{p^{r+1}-1}{p-1}}}
\equiv 0\pmod{p^{2r+5}}.
\end{equation}
\end{conjecture}

\begin{conjecture}\label{c5}
Let $p>3$ be a prime. Then
\begin{equation}
\sum_{k=0}^{p-1}(2pk-2k-1)\frac{(\frac{-1}{p-1})_k^{2p-2}}{k!^{2p-2}}
\equiv 0\pmod{p^5}. \label{eq:p5}
\end{equation}
More generally, if $p>3$ is a prime and $r$ a positive integer, then
\begin{equation}
\sum_{k=0}^{p^r-1}(2kp^r-2k-1)\frac{(\frac{-1}{p^r-1})_k^{2p^r-2}}{k!^{2p^r-2}}
\equiv 0\pmod{p^{2r+3}}.
\end{equation}
\end{conjecture}

Conjectures~\ref{c7} and \ref{c5} are quite remarkable
as they concern supercongruences modulo high prime powers.
We now give two partial $q$-analogues of \eqref{eq:p7} as follows.

\begin{conjecture}
Let $n$ be an integer greater than $1$. Then
\begin{equation*}
\sum_{k=0}^{n-1}
[2nk+2k+1]\frac{(q;q^{n+1})_k^{2n+2}}{(q^{n+1};q^{n+1})_k^{2n+2}}
q^{(n+1)(n-1)k}\equiv 0\pmod{[n]^2\Phi_n(q)^2\Phi_{n^2}(q)}.
\end{equation*}
\end{conjecture}

\begin{conjecture}
Let $p$ be a prime. Then
\begin{align}\label{conj2}
&\sum_{k=0}^{p-1}
[2pk+2k+1]\frac{(q;q^{p+1})_k^{2p+2}}{(q^{p+1};q^{p+1})_k^{2p+2}}
q^{(p+1)(p-1)k}\notag\\
&\equiv
-\frac{(2p+1)(p+1)^2p(p-1)}{72}q(1-q)^2[p]^4\Phi_{p^2}(q)
\pmod{[p]^5\Phi_{p^2}(q)}.
\end{align}
\end{conjecture}

It is clear that the $q\to 1$ case of \eqref{conj2} reduces to
\eqref{eq:p7} modulo $p^6$. We would like to emphasize that
\eqref{conj2}, while still conjectural, appears to be the first example of
a basic hypergeometric supercongruence in the existing literature,
that in the limit $q\to 1$ reduces to a supercongruence
(for a hypergeometric series being truncated after a number of terms
that is linear in $p$) modulo $p^6$.

Finally, we give a partial and a complete $q$-analogue of \eqref{eq:p5}
as follows.

\begin{conjecture}
Let $n$ be an integer greater than $1$. Then
\begin{equation*}
\sum_{k=0}^{n-1}
[2nk-2k-1]\frac{(q^{-1};q^{n-1})_k^{2n-2}}{(q^{n-1};q^{n-1})_k^{2n-2}}
q^{(n-1)^2k}\equiv 0\pmod{[n]^2\Phi_n(q)^2}.
\end{equation*}
\end{conjecture}

\begin{conjecture}
Let $p$ be a prime. Then
\begin{align}\label{conj3}
&\sum_{k=0}^{p-1}
[2pk-2k-1]\frac{(q^{-1};q^{p-1})_k^{2p-2}}{(q^{p-1};q^{p-1})_k^{2p-2}}
q^{(p-1)^2k}\notag\\
&\equiv
\frac{(2p-3)(p-1)(p-2)^2(p-3)}{6}(1-q)^2[p]^4
\pmod{[p]^5}.
\end{align}
\end{conjecture}

It is clear that the $q\to 1$ case of \eqref{conj3} reduces to the
modulo $p^5$ congruence in \eqref{eq:p5}.

\end{document}